\documentclass[reqno,b5paper]{amsart}
\usepackage{amsfonts}

\usepackage{amsmath}
\usepackage{amssymb}
\usepackage{amsthm}
\usepackage{enumerate}
\usepackage[mathscr]{eucal}
\usepackage{eqlist}

\theoremstyle{plain}
\newtheorem{theorem}{Theorem}[section]

\theoremstyle{definition}
\newtheorem{definition}{Definition}[section]

\theoremstyle{remark}
\newtheorem{example}{Example}
\numberwithin{equation}{section}

\input{tcilatex}

\begin{document}
\title[Well posed neutral difference equations]{On the well posed solutions
for nonlinear second order neutral difference equations}
\author{Marek Galewski*}
\author[Marek Galewski \and Ewa Schmeidel]{Ewa Schmeidel**}
\address{\llap{*\,}Institute of Mathematics\\
\indent Lodz University of Technology\\
\indent Wolczanska 215\\
\indent94-011 Lodz\\
\indent POLAND}
\email{marek.galewski@p.lodz.pl}
\address{\llap{**\,}Faculty of Mathematics and Computer Science\\
\indent University of Bialystok\\
\indent Akademicka 2\\
\indent15-267 Bialystok\\
\indent POLAND}
\email{eschmeidel@math.uwb.edu.pl}

\begin{abstract}
In this work we investigate the existence of solutions, their uniqueness and
finally dependence on parameters for solutions of second order neutral
nonlinear difference equations. The main tool which we apply is Darbo fixed
point theorem.
\end{abstract}

\maketitle

\section{Introduction}

In this paper we consider a second order neutral difference equation 
\begin{equation}
\Delta \left( r_{n}\left( \Delta \left( x_{n}+p_{n}x_{n-k}\right) \right)
^{\gamma _{n}}\right) +q_{n}x_{n}^{\alpha }+a_{n}f(x_{n+1})=0,  \label{e0}
\end{equation}%
with variable exponent being the ratio of odd positive integers. To be
precise, $\left( \gamma _{n}\right) $ is the sequence consisting of the
ratio of odd positive integers, $\alpha \geq 1$, $x:N_{0}\rightarrow R$, $%
a,p,q:N_{0}\rightarrow R,$ $r:N_{0}\rightarrow R\setminus \{0\}$ are
sequences and $f:R\rightarrow R$ is a locally Lipschitz function satisfying
no further growth assumptions. We denote $N_{0}:=\left\{ 0,1,2,\dots
\right\} $, $N_{k}:=\left\{ k,k+1,\dots \right\} $ where $k$ is a given
positive integer, and $R$ is a set of all real numbers. By a solution of
equation~\eqref{e0} we mean a sequence $x:N_{k}\rightarrow R$ which
satisfies~\eqref{e0} for every $n\in \ N_{k}$.

Since difference equations arise naturally in diverse areas, such as
economy, biology, physics, mechanics, computer science, finance, see for
example \cite{bib1B}, \cite{bib2B}, it is of interest to know conditions
which guarantee

a) existence of solutions,

b) uniqueness,

c) dependence of the solutions on parameters.

Problems satisfying all three conditions are called well--posed. Such
investigations have been undertaken for two point discrete boundary value
problems based in a finite dimensional space, see for example \cite{MG},
where the variational methods are applied. However, to the best of the
authors knowledge, similar results have not been obtained for second order
neutral difference equations. The approach by critical point theory allows
for investigation connected with dependence on parameters without having
unique solution. We are unable to adopt this approach since in $L^{\infty }$
one cannot select weakly convergent subsequences from all bounded sequences.
Moreover, the new idea used in this paper is to consider \eqref{e0} with
variable the ratio of odd positive integers, contrary to earlier results by
these authors \cite{MGES}, where this ratio is held fixed.

Problem \eqref{e0} can be seen as a generalized version of the Emden--Fowler
equation which originated in the gaseous dynamics in astrophysics and
further was used in the study of fluid mechanics, relativistic mechanics,
nuclear physics and in the study of chemically reacting systems, see \cite%
{wong}. Moreover, taking $p\equiv 0$, $q\equiv 0$ and $f(x)=x$ we get a type
of a Sturm--Liouville difference equation investigated for example by Do\v{s}%
l \'{y} in \cite{dosly1}, \cite{DOSLY2}, \cite{DOSLY3}.

A number of the mathematical models described by discrete equations can be
found in numerous well--known monographs: Agarwal~\cite{bib1B}, Agarwal,
Bohner, Grace and O'Regan \cite{ABGO}, Agarwal and Wong \cite{AW}, Elaydi~ 
\cite{bib2B}, Koci\'{c} and Ladas~\cite{bib3B} and Peterson~\cite{bib4B}.

We mention that there has already been some interest in properties of
solutions of second order difference equations: see, for example Bajo and
Liz \cite{B}, Koci\'{c} and Ladas \cite{KocicLadas}, Qian and Yan \cite%
{LadasQian}, Migda and Migda \cite{bib2}, Migda, Schmeidel and Zb\c{a}%
szyniak~\cite{bib3}, Schmeidel \cite{bib5}, and Thandapani, Kavitha and
Pinelas \cite{TKP1}--\cite{TKP2}.

Paper is organized as follows. Firstly, we provide necessary preliminaries
on the measure of noncompactness. Next we derive existence result. Further,
assuming that $\alpha =1$ we investigate a special type stability and the
uniqueness of a solution which seem to follow the same proof. Finally, we
consider dependence on parameters. Each further step in our investigations
restricts the type of nonlinear functions which can be taken into account.

\section{Preliminaries}

In this paper, we will use axiomatically defined measures of noncompactness
as presented in paper~\cite{bib1} by Bana\'{s} and Rzepka which we now
recall for reader's convenience (see also \cite{SchmeildelZbaszyniakCAMW}).
We also refer to the monograph \cite{bibM1B}. Let $(E,\left\Vert \cdot
\right\Vert )$ be an infinite--dimensional Banach space. If $X$ is a subset
of $E$, then $\bar{X}$, $ConvX$ denote the closure and the convex closure of 
$X$, respectively. Moreover, we denote by $M_{E}$ the family of all nonempty
and bounded subsets of $E$ and by $N_{E}$ denote the subfamily consisting of
all relatively compact sets.

\begin{definition}
A mapping $\mu \colon M_{E}\rightarrow \lbrack 0,\infty )$ is called a
measure of noncompactness in $E$ if it satisfies the following conditions:

\begin{description}
\item[$1^{0}$] $\ker \mu =\left\{ X\in {\mathcal{M}}_{E}\colon \mu
(X)=0\right\} \neq \emptyset \mbox{  and  }\ker \mu \subset N_{E},$

\item[$2^{0}$] $X\subset Y\Rightarrow \mu (X)\leq \mu (Y),$

\item[$3^{0}$] $\mu (\bar{X})=\mu (X)=\mu (Conv\,\,X),$

\item[$4^{0}$] $\mu (\alpha X+(1-\alpha )Y)\leq \alpha \mu (X)+(1-\alpha
)\mu (Y)\mbox{  for  }0\leq \alpha \leq 1,$

\item[$5^{0}$] If $X_{n}\in M_{E},\,\,\,X_{n+1}\subset X_{n},\,\,\,X_{n}=%
\bar{X_{n}}\mbox{  for  }n=1,2,3,\dots $

and $\lim\limits_{n\rightarrow \infty }\mu (X_{n})=0\mbox{  then  }%
\bigcap\limits_{n=1}^{\infty }X_{n}\neq \emptyset $.
\end{description}
\end{definition}

The following Darbo's fixed point theorem given in~\cite{bib1} is used in
the proof of the main result.

\begin{theorem}
\label{D} Let $M$ be a nonempty, bounded, convex and closed subset of the
space $E$ and let $T:M\rightarrow M$ be a continuous operator such that $\mu
(T(X))\leq k\mu (X)$ for all nonempty subset $X$ of $M$, where $k\in \lbrack
0,1)$ is a constant. Then $T$ has a fixed point in the subset $M$.
\end{theorem}

We consider the Banach space $l^{\infty }$ of all real bounded sequences $%
x\colon N_{0}\rightarrow R$ equipped with the standard supremum norm, i.e. 
\begin{equation*}
\Vert x\Vert =\sup\limits_{n\in {\mathbb{N}}_{0}}|x_{n}|\text{ for }x\in
l^{\infty }.
\end{equation*}%
Let $X$ be a nonempty, bounded subset of $l^{\infty }$, $X_{n}=\left\{
x_{n}:x\in X\right\} $ (it means $X_{n}$ is a set of $n$--th terms of any
sequence belonging to $X$), and 
\begin{equation*}
diam\,\,X_{n}=\sup \left\{ \left\vert x_{n}-y_{n}\right\vert :x,y\in
X\right\} .
\end{equation*}%
We use a following measure of noncompactness in the space $l^{\infty }$ (see 
\cite{bibM1B}) 
\begin{equation*}
\mu (X)=\limsup_{n\rightarrow \infty }diam\,\,X_{n}.
\end{equation*}

\section{Existence result}

In this section, sufficient conditions for the existence of a bounded
solution of equation~\eqref{e0} are derived.

\begin{theorem}
\label{L2}Let the sequence $(\gamma _{n})$ of the ratio of odd positive
integers be such that 
\begin{equation}
\gamma ^{+}=\sup_{n\in {\mathbb{N}}}\gamma _{n}\leq 1.  \label{gamma}
\end{equation}%
Let a number $\alpha \geq 1$ be fixed and let $k$ be a fixed positive
integer. Assume that 
\begin{equation*}
f:{\mathbb{R}}\rightarrow {\mathbb{R}}\text{ is a locally Lipschitz function,%
}
\end{equation*}%
the sequences $r:N_{0}\rightarrow R\setminus \{0\}$, $a,q:N_{0}\rightarrow R$
satisfy 
\begin{equation}
\sum\limits_{n=0}^{\infty }\left[ \left\vert \frac{1}{r_{n}}\right\vert
\sum\limits_{i=n}^{\infty }\left\vert a_{i}\right\vert \right] ^{\frac{1}{%
\gamma ^{+}}}<+\infty ,  \label{z2}
\end{equation}%
and 
\begin{equation}
\sum\limits_{n=0}^{\infty }\left[ \left\vert \frac{1}{r_{n}}\right\vert
\sum\limits_{i=n}^{\infty }\left\vert q_{i}\right\vert \right] ^{\frac{1}{%
\gamma ^{+}}}\leq +\infty .  \label{z22}
\end{equation}%
Assume that the sequence $p:N_{0}\rightarrow R\setminus \{0\}$ satisfies the
following condition 
\begin{equation}
-1<\liminf\limits_{n\rightarrow \infty }p_{n}\leq
\limsup\limits_{n\rightarrow \infty }p_{n}<1.  \label{z3}
\end{equation}%
Assume additionally that 
\begin{equation}
\sum\limits_{i=0}^{\infty }\left\vert a_{i}\right\vert <+\infty
,\,\,\sum\limits_{i=0}^{\infty }\left\vert q_{i}\right\vert <+\infty .
\label{add_series}
\end{equation}%
Then, there exists a bounded solution $x:N_{k}\rightarrow R$ of equation~%
\eqref{e0}.
\end{theorem}

\begin{proof}
Condition \eqref{z3} implies that there exist $n_{1}\in \ N_{0}$ and a
constant $P\in \lbrack 0,1)$ such that 
\begin{equation}
\left\vert p_{n}\right\vert \leq P<1,\text{ for }n\geq n_{1}.  \label{z5}
\end{equation}%
Recalling that the remainder of a series is the difference between the sum
of a series and its partial sum, we denote by $\alpha _{n}$ and by $\beta
_{n}$ remainders of both series in \eqref{z2} and \eqref{z22}. This means
that 
\begin{equation*}
\alpha _{n}=\sum\limits_{j=n}^{\infty }\left[ \left\vert \frac{1}{r_{j}}%
\right\vert \sum\limits_{i=j}^{\infty }\left\vert a_{i}\right\vert \right] ^{%
\frac{1}{\gamma _{i}}}\text{ and }\,\,\beta _{n}=\sum\limits_{j=n}^{\infty }%
\left[ \left\vert \frac{1}{r_{j}}\right\vert \sum\limits_{i=j}^{\infty
}\left\vert q_{i}\right\vert \right] ^{\frac{1}{\gamma _{i}}}
\end{equation*}%
and by \eqref{z2} and \eqref{z22}, we see that 
\begin{equation*}
\lim\limits_{n\rightarrow \infty }\alpha _{n}=0\,\,\text{ and }%
\lim\limits_{n\rightarrow \infty }\beta _{n}=0.
\end{equation*}%
Fix any number $d>0$. There exists a constant $M_{d}>0$ such that $%
\left\vert f\left( x\right) \right\vert \leq M_{d}$ for all $x\in \left[ -d,d%
\right] $. Set 
\begin{equation}
M^{\ast }=\max \left\{ M_{d},d^{\alpha }\right\} .  \label{max}
\end{equation}%
Chose a constant $C$ such that 
\begin{equation*}
0<C\leq \frac{d-Pd}{(2M^{\ast })^{\frac{1}{\gamma ^{+}}}}.
\end{equation*}%
Let us define sequence $(m_{n})$ as follows 
\begin{equation*}
m_{n}=\max \{\left\vert a_{n}\right\vert ,\left\vert q_{n}\right\vert \}.
\end{equation*}%
From the above and by conditions \eqref{z2}, \eqref{z22} and %
\eqref{add_series}, there exists a positive integer $n_{2}$ such that 
\begin{equation}
\alpha _{n}\leq C,\,\,\beta _{n}\leq C,\,\,\sum\limits_{i=n}^{\infty
}\left\vert a_{i}\right\vert \leq 1\text{ and }\sum\limits_{i=n}^{\infty
}\left\vert q_{i}\right\vert \leq 1  \label{z8}
\end{equation}%
for $n\geq n_{2}$, $n\in N_{n_{2}}\colon =\left\{
n_{2},n_{2}+1,n_{2}+2,\dots \right\} $. Define also $N_{n_{3}}\colon
=\left\{ n_{3},n_{3}+1,n_{3}+2,\dots \right\} $ and $n_{3}=\max \left\{
n_{1},n_{2}\right\} $.

Define set $B$ as follows 
\begin{equation}
B\colon =\left\{ (x_{n})_{n=0}^{\infty }:\left\vert x_{n}\right\vert \leq d,%
\text{ for }n\in {\mathbb{N}}_{0}\right\} .  \label{z9}
\end{equation}%
Observe that $B$ is a nonempty, bounded, convex and closed subset of $%
l^{\infty }$.

Define a mapping $T\colon B\rightarrow l^{\infty }$ as follows 
\begin{equation}
(Tx)_{n}=%
\begin{cases}
x_{n}, & \text{for any}\ 0\leq n<n_{3} \\ 
-p_{n}x_{n-k}- & \sum\limits_{j=n}^{\infty }\left[ \frac{1}{r_{j}}%
\sum\limits_{i=j}^{\infty }\left( a_{i}f(x_{i+1})+q_{i}x_{i}^{\alpha
}\right) \right] ^{\frac{1}{\gamma _{j}}},\text{ for }\ n\geq n_{3}.%
\end{cases}
\label{z10}
\end{equation}%
We see that $T$ is a continuous operator. We will prove that the mapping $T$
has a fixed point in $B$. This proof will follow in several subsequent steps.

\bigskip Step 1. Firstly, we show that $T(B)\subset B$.

If $x\in B$, then by~\eqref{z10}, \eqref{z5}, \eqref{z9}, \eqref{max} and %
\eqref{z8}, we have for $n\geq n_{3}$ 
\begin{equation*}
\begin{array}{l}
\left\vert (Tx)_{n}\right\vert \,\,\leq \left\vert p_{n}\right\vert
\left\vert x_{n-k}\right\vert +\left\vert \sum\limits_{j=n}^{\infty }\left[ 
\frac{1}{r_{j}}\sum\limits_{i=j}^{\infty }\left(
a_{i}f(x_{i+1})+q_{i}x_{i}^{\alpha }\right) \right] ^{\frac{1}{\gamma _{j}}%
}\right\vert \bigskip  \\ 
\leq \left\vert p_{n}\right\vert \left\vert x_{n-k}\right\vert
+\sum\limits_{j=n}^{\infty }\left[ \left\vert \frac{1}{r_{j}}%
\sum\limits_{i=j}^{\infty }\left( a_{i}f(x_{i+1})+q_{i}x_{i}^{\alpha
}\right) \right\vert \right] ^{\frac{1}{\gamma _{j}}}\bigskip  \\ 
\leq \left\vert p_{n}\right\vert \left\vert x_{n-k}\right\vert
+\sum\limits_{j=n}^{\infty }\left[ \left\vert \frac{1}{r_{j}}\right\vert
\sum\limits_{i=j}^{\infty }\left( \left\vert a_{i}\right\vert \left\vert
f(x_{i+1})\right\vert +\left\vert q_{i}\right\vert \left\vert
x_{i}\right\vert ^{\alpha }\right) \right] ^{\frac{1}{\gamma _{j}}}\bigskip 
\end{array}%
\end{equation*}%
\begin{equation*}
\begin{array}{l}
\leq \left\vert p_{n}\right\vert \left\vert x_{n-k}\right\vert
+\sum\limits_{j=n}^{\infty }\left[ M^{\ast }\left\vert \frac{1}{r_{j}}%
\right\vert \sum\limits_{i=j}^{\infty }\left( \left\vert a_{i}\right\vert
+\left\vert q_{i}\right\vert \right) \right] ^{\frac{1}{\gamma _{j}}%
}\bigskip  \\ 
\leq \left\vert p_{n}\right\vert \left\vert x_{n-k}\right\vert
+\sum\limits_{j=n}^{\infty }\left[ 2M^{\ast }\left\vert \frac{1}{r_{j}}%
\right\vert \sum\limits_{i=j}^{\infty }m_{i}\right] ^{\frac{1}{\gamma _{j}}%
}\bigskip  \\ 
\leq \left\vert p_{n}\right\vert \left\vert x_{n-k}\right\vert
+\sum\limits_{j=n}^{\infty }\left[ 2M^{\ast }\left\vert \frac{1}{r_{j}}%
\right\vert \sum\limits_{i=j}^{\infty }m_{i}\right] ^{\frac{1}{\gamma ^{+}}%
}\bigskip  \\ 
\leq \left\vert p_{n}\right\vert \left\vert x_{n-k}\right\vert +(2M^{\ast
})^{\frac{1}{\gamma ^{+}}}\sum\limits_{j=n}^{\infty }\left[ \left\vert \frac{%
1}{r_{j}}\right\vert \sum\limits_{i=j}^{\infty }m_{i}\right] ^{\frac{1}{%
\gamma ^{+}}}\bigskip  \\ 
\leq Pd+(2M^{\ast })^{\frac{1}{\gamma ^{+}}}\max \left\{ \alpha _{n},\beta
_{n}\right\} \leq Pd+(2M^{\ast })^{\frac{1}{\gamma ^{+}}}C.%
\end{array}%
\end{equation*}%
From the above, we obtain 
\begin{equation}
\left\vert (Tx)_{n}\right\vert \leq Pd+(2M^{\ast })^{\frac{1}{\gamma ^{+}}}%
\frac{d-Pd}{(2M^{\ast })^{\frac{1}{\gamma ^{+}}}}=d  \label{calc}
\end{equation}%
for $n\in N_{n_{3}}$ and obviously $\left\vert (Tx)_{n}\right\vert
=\left\vert x_{n}\right\vert \leq d$ when $n<n_{3}$.

\bigskip Step 2. $T$\ is continuous.

By assumption (\ref{add_series}) and by definition of $B$ there exists a
constant $c>0$ such that 
\begin{equation*}
\sum\limits_{i=j}^{\infty }\left\vert a_{i}f(x_{i+1})+q_{i}x_{i}^{\alpha
}\right\vert \leq c,\,\,\,j\in \mathbb{N}_{n_{3}},
\end{equation*}%
for all $x\in B$. Since $t\rightarrow t^{1/\gamma _{n}}$ is differentiable
for each fixed $n$, it is Lipschitz on closed and bounded intervals. By
Lagrange Mean Value Theorem, the Lipschitz constant satisfies the following
estimation $L_{\gamma ^{+}}\leq \frac{1}{\gamma ^{+}}d^{\frac{1}{\gamma ^{+}}%
-1}$. Thus we have 
\begin{equation*}
\left\vert t^{1/\gamma _{n}}-s^{1/\gamma _{n}}\right\vert \leq L_{\gamma
^{+}}\left\vert t-s\right\vert \text{ for all }t,s\in \left[ -c,c\right] .
\end{equation*}%
Since $f$ is locally Lipschitz it is Lipschitz on $\left[ -d,d\right] $. So
there is a constant $L_{d}>0$ such that 
\begin{equation*}
\left\vert f\left( x\right) -f\left( y\right) \right\vert \leq
L_{d}\left\vert x-y\right\vert 
\end{equation*}%
for all $x,y\in \left[ -d,d\right] $. Since $x\rightarrow x^{\alpha }$ is
also Lipschitz on $\left[ -d,d\right] $, there is a constant $L_{\alpha }$
such that 
\begin{equation*}
\left\vert x^{\alpha }-y^{\alpha }\right\vert \leq L_{\alpha }\left\vert
x-y\right\vert \text{ for all }x,y\in \left[ -d,d\right] .
\end{equation*}

Let $y^{(p)}$ be a sequence in $B$ such that $\left\Vert
x^{(p)}-x\right\Vert \rightarrow 0$ as $p\rightarrow \infty $. Since $B$ is
closed, $x\in B$. Now we will examine the convergence of $Ty^{(p)}$. The
reasoning is by definition of $T$ and is similar to the arguments which lead
to formula \eqref{z10}. By \eqref{z2}, \eqref{z22}, \eqref{calc}, we get for 
$n\geq n_{3}$ 
\begin{equation*}
\begin{array}{l}
\left\vert (Tx)_{n}-(Ty^{(p)})_{n}\right\vert \,\,\leq \left\vert
p_{n}\right\vert \left\vert x_{n-k}-y_{n-k}^{(p)}\right\vert \bigskip  \\ 
+\sum\limits_{j=n}^{\infty }\left\vert \frac{1}{r_{j}}\right\vert ^{\frac{1}{%
\gamma ^{+}}}\left\vert \left[ \sum\limits_{i=j}^{\infty }\left(
a_{i}f(x_{i+1})+q_{i}\left( x_{i}\right) ^{\alpha }\right) \right] ^{\frac{1%
}{\gamma _{j}}}-\left[ \sum\limits_{i=j}^{\infty }\left(
a_{i}f(y_{i+1}^{(p)})+q_{i}\left( y_{i}^{\left( p\right) }\right) ^{\alpha
}\right) \right] ^{\frac{1}{\gamma _{j}}}\right\vert 
\end{array}%
\end{equation*}%
\begin{equation*}
\begin{array}{l}
\leq \,\,\,\left\vert p_{n}\right\vert \left\vert
x_{n-k}-y_{n-k}^{(p)}\right\vert  \\ 
+\sum\limits_{j=n}^{\infty }\left\vert \frac{1}{r_{j}}\right\vert ^{\frac{1}{%
\gamma ^{+}}}L_{\gamma ^{+}}\left\vert \sum\limits_{i=j}^{\infty
}a_{i}f(x_{i+1})+\sum\limits_{i=j}^{\infty }q_{i}\left( x_{i}\right)
^{\alpha }-\sum\limits_{i=j}^{\infty }a_{i}f(y_{i+1}^{\left( p\right)
})-\sum\limits_{i=j}^{\infty }q_{i}\left( y_{i}^{\left( p\right) }\right)
^{\alpha }\right\vert \bigskip 
\end{array}%
\end{equation*}%
\begin{equation*}
\begin{array}{l}
\leq \left\vert p_{n}\right\vert \left\vert x_{n-k}-y_{n-k}^{(p)}\right\vert
+L_{\gamma ^{+}}\sum\limits_{j=n}^{\infty }\left\vert \frac{1}{r_{j}}%
\right\vert ^{\frac{1}{\gamma ^{+}}}\sum\limits_{i=j}^{\infty }\left\vert
a_{i}\right\vert \left\vert f(x_{i+1})-f(y_{i+1}^{\left( p\right)
})\right\vert \bigskip  \\ 
+L_{\gamma }\sum\limits_{j=n}^{\infty }\left\vert \frac{1}{r_{j}}\right\vert
^{\frac{1}{\gamma }}\sum\limits_{i=j}^{\infty }\left\vert q_{i}\right\vert
\left\vert \left( x_{i}\right) ^{\alpha }-\left( y_{i}^{\left( p\right)
}\right) ^{\alpha }\right\vert \bigskip \leq \left\vert p_{n}\right\vert
\left\vert x_{n-k}-y_{n-k}^{(p)}\right\vert \bigskip  \\ 
+L_{\gamma ^{+}}L_{d}\sum\limits_{j=n}^{\infty }\left\vert \frac{1}{r_{j}}%
\right\vert ^{\frac{1}{\gamma ^{+}}}\sum\limits_{i=j}^{\infty }\left\vert
a_{i}\right\vert \left\vert x_{i+1}-y_{i+1}^{\left( p\right) }\right\vert
+L_{\gamma ^{+}}L_{\alpha }\sum\limits_{j=n}^{\infty }\left\vert \frac{1}{%
r_{j}}\right\vert ^{\frac{1}{\gamma ^{+}}}\sum\limits_{i=j}^{\infty
}\left\vert q_{i}\right\vert \left\vert x_{i}-y_{i}^{\left( p\right)
}\right\vert \bigskip 
\end{array}%
\end{equation*}%
\begin{equation*}
\begin{array}{l}
\leq \left( \left\vert p_{n}\right\vert +L_{\gamma
^{+}}L_{d}\sum\limits_{j=n}^{\infty }\left\vert \frac{1}{r_{j}}\right\vert ^{%
\frac{1}{\gamma ^{+}}}\sum\limits_{i=j}^{\infty }\left\vert a_{i}\right\vert
+L_{\gamma ^{+}}L_{\alpha }\sum\limits_{j=n}^{\infty }\left\vert \frac{1}{%
r_{j}}\right\vert ^{\frac{1}{\gamma ^{+}}}\sum\limits_{i=j}^{\infty
}\left\vert q_{i}\right\vert \right) \left\Vert y^{(p)}-x\right\Vert .%
\end{array}%
\end{equation*}%
$\left\vert (Tx)_{n}-(Ty^{(p)})_{n}\right\vert \,=0$ for $n<n_{3}$.

Thus 
\begin{equation*}
\lim\limits_{p\rightarrow \infty }\left\Vert Ty^{(p)}-Tx\right\Vert =0\text{
as }\lim\limits_{p\rightarrow \infty }\left\Vert y^{(p)}-x\right\Vert =0.
\end{equation*}%
This means that $T$ is continuous.

\bigskip Step 3. Comparison of the measure of noncompactness

Now, we need to compare a measure of noncompactness of any subset $X$ of $B$
and $T(X)$. Let us fix any nonempty set $X\subset B$. Take any sequences $%
x,y\in X$. Following the same calculations which led to the continuity of
the operator $T$ we see that 
\begin{equation*}
\left\vert (Tx)_{n}-(Ty)_{n}\right\vert \leq \left( \left\vert
p_{n}\right\vert +L_{\gamma ^{+}}L_{\alpha }\beta _{n}\right) \left\vert
x_{n}-y_{n}\right\vert +L_{\gamma ^{+}}L_{d}\alpha _{n}\left\vert
x_{n+1}-y_{n+1}\right\vert 
\end{equation*}%
Taking sufficiently large $n$ we get 
\begin{equation*}
\left\vert p_{n}\right\vert +L_{\gamma ^{+}}L_{d}\alpha _{n}\leq c_{1}<\frac{%
1+P}{2}
\end{equation*}%
and 
\begin{equation*}
L_{\gamma ^{+}}L_{\alpha }\beta _{n}\leq c_{2}<\frac{1-P}{2}
\end{equation*}%
since $\alpha _{n}\rightarrow 0$ and $\beta _{n}\rightarrow 0$, where $%
c_{1},c_{2}$ are some constants. We see that 
\begin{equation*}
diam\,\,(T(X))_{n}\leq c_{1}diam\,\,X_{n}+c_{2}diam\,\,X_{n+1}.
\end{equation*}%
This yields by the properties of the upper limit that 
\begin{equation*}
\limsup_{n\rightarrow \infty }diam\,\,(T(X))_{n}\leq
c_{1}\,\limsup_{n\rightarrow \infty
}diam\,\,X_{n}+c_{2}\limsup_{n\rightarrow \infty }diam\,\,X_{n+1}.
\end{equation*}%
From above, for any $X\subset B$ we have $\mu (T(X))\leq \left(
c_{1}\,+c_{2}\right) \mu (X)$.

\bigskip Step 4. Relation between fixed points and solutions.

By Theorem~\ref{D} we conclude that $T$ has a fixed point in the set $B$. It
means that there exists $x\in B$ such that 
\begin{equation*}
x_{n}=(Tx)_{n}.
\end{equation*}%
Thus 
\begin{equation*}
x_{n}=-p_{n}x_{n-k}-\sum\limits_{j=n}^{\infty }\left[ \frac{1}{r_{j}}%
\sum\limits_{i=j}^{\infty }\left( a_{i}f(x_{i+1})+q_{i}x_{i}^{\alpha
}\right) \right] ^{\frac{1}{\gamma _{j}}},\text{ for }\,\,n\in {\mathbb{N}}%
_{n_{3}}.
\end{equation*}%
We will examine a correspondence between fixed points of $T$\ and solutions
to \eqref{e0}. We apply operator $\Delta $\ to both sides of the following
equation 
\begin{equation*}
x_{n}+p_{n}x_{n-k}=-\sum\limits_{j=n}^{\infty }\left[ \frac{1}{r_{j}}%
\sum\limits_{i=j}^{\infty }\left( a_{i}f(x_{i+1})+q_{i}x_{i}^{\alpha
}\right) \right] ^{\frac{1}{\gamma _{j}}},\,\,\,n\in {\mathbb{N}}_{n_{3}}.
\end{equation*}%
We find that 
\begin{equation*}
\Delta (x_{n}+p_{n}x_{n-k})=\left[ \frac{1}{r_{n}}\sum\limits_{i=n}^{\infty
}\left( a_{i}f(x_{i+1})+q_{i}x_{i}^{\alpha }\right) \right] ^{\frac{1}{%
\gamma _{n}}},\,\,\,n\in {\mathbb{N}}_{n_{3}}.
\end{equation*}%
and next 
\begin{equation*}
r_{n}\left( \Delta (x_{n}+p_{n}x_{n-k})\right) ^{\gamma
_{n}}=\sum\limits_{i=n}^{\infty }\left( a_{i}f(x_{i+1})+q_{i}x_{i}^{\alpha
}\right) ,\,\,\,n\in {\mathbb{N}}_{n_{3}}.
\end{equation*}%
Taking operator $\Delta $ again to both sides of the above equation we
obtain 
\begin{equation*}
\Delta \left( r_{n}\left( \Delta (x_{n}+p_{n}x_{n-k})\right) ^{\gamma
_{n}}\right) =-a_{n}f(x_{n+1})-q_{n}x_{n}^{\alpha },\,\,\,n\in {\mathbb{N}}%
_{n_{3}}.
\end{equation*}%
Hence, we get equation~\eqref{e0}.\ Sequence $x$, which is a fixed point of
mapping $T$, is a bounded sequence which fulfills equation~\eqref{e0} for
large $n$. If $n_{3}\leq k$\ the proof is ended. If $n_{3}>k$\ we find
previous $n_{3}-k+1$\ terms of sequence $x$\ by formula 
\begin{equation*}
x_{n-k+l}=\frac{1}{p_{n+l}}\left( -x_{n+l}+\sum\limits_{j={n+l}}^{\infty }%
\left[ \frac{1}{r_{j}}\sum\limits_{i=j}^{\infty }\left(
a_{i}f(x_{i+1})+q_{i}x_{i}^{\alpha }\right) \right] ^{\frac{1}{\gamma _{j}}%
}\right) ,
\end{equation*}%
where $l\in \left\{ 0,1,2,\dots ,k-1\right\} $, which result leads directly
from~\eqref{e0}. It means that equation~\eqref{e0} has at least one bounded
solution $x:N_{k}\rightarrow R$. This completes the proof.
\end{proof}

One remark is in order as concerns the proof.

\textbf{Remark. }\textit{We see that cannot go beyond }$n=k$\textit{\ in
this iteration, so we do not use the whole fixed point solution. We also
recall that a fixed point of operator }$T$\textit{\ does not provide a
solution but the solution is obtained through a fixed point method combined
with backward iteration. When we do not have dependence on }$k$\textit{\ in
the main equation, we have direct and obvious correspondence between fixed
point of }$T$\textit{\ and solutions to \eqref{e0}.}

\section{A special type of stability}

In this section we derive sufficient conditions for the existence of some
type of a stable solution of equation~\eqref{e0}. We describe the stability
which we mean in the content of the theorem.

\begin{theorem}
\label{T1}Let the sequence $(\gamma _{n})$ of the ratio of odd positive
integers be such that \eqref{gamma} holds. Assume that $\alpha \geq 1$ is
fixed. Let $k$ be a fixed positive integer. Assume that 
\begin{equation*}
f:{\mathbb{R}}\rightarrow {\mathbb{R}}\text{ is a locally Lipschitz function}
\end{equation*}%
and that conditions \eqref{z2}--\eqref{add_series} hold. Then equation~%
\eqref{e0} has at least one solution $x:N_{k}\rightarrow R$ with the
following stability property: given any other bounded solution $%
y:N_{k}\rightarrow R$ and $\varepsilon >0$ there exists $T>$ $n_{3}$ such
that for every $t\geq T$ the following inequality holds 
\begin{equation*}
\left\vert x(t)-y(t)\right\vert \leq \varepsilon .
\end{equation*}
\end{theorem}

\begin{proof}
From Theorem~\ref{L2}, equation~\eqref{e0} has at least one bounded solution 
$x:N_{0}\rightarrow R$ which can be rewritten in the form 
\begin{equation*}
x_{n}=(Tx)_{n},\text{ for }n\geq n_{3}
\end{equation*}%
where mapping $T$ is defined by~\eqref{z10}. Let $y$ be another bounded
solution of equation~\eqref{e0}. Then there exists a constant $d$ such that $%
\sup_{t\in N_{0}}\left\vert x\left( t\right) -y(t)\right\vert \leq d$. Since 
$f$ is locally Lipschitz it is Lipschitz on $\left[ -d,d\right] $ with a
Lipschitz constant $L_{d}>0$. Similar arguments concern function $%
x\rightarrow x^{\alpha }$ which is also Lipschitz on $\left[ -d,d\right] $
with a Lipschitz constant $L_{\alpha }>0$. Following the steps of the proof
of Theorem \ref{L2} we see that using the fact that 
\begin{equation}
\begin{array}{l}
\left\vert x_{n}-y_{n}\right\vert =\left\vert (Tx)_{n}-(Ty)_{n}\right\vert
\,\bigskip  \\ 
\leq \left\vert p_{n}\right\vert \left\vert x_{n-k}-y_{n-k}\right\vert
+L_{\gamma ^{+}}L_{d}\sum\limits_{j=n}^{\infty }\left\vert \frac{1}{r_{j}}%
\right\vert ^{\frac{1}{\gamma ^{+}}}\sum\limits_{i=j}^{\infty }\left\vert
a_{i}\right\vert \left\vert x_{i+1}-y_{i+1}\right\vert \,\bigskip  \\ 
\,+L_{\gamma ^{+}}L_{\alpha }\sum\limits_{j=n}^{\infty }\left\vert \frac{1}{%
r_{j}}\right\vert ^{\frac{1}{\gamma ^{+}}}\sum\limits_{i=j}^{\infty
}\left\vert q_{i}\right\vert \left\vert x_{i}-y_{i}\right\vert .%
\end{array}
\label{arg}
\end{equation}%
Note that for $n$ large enough, say $n\geq n_{4}\geq n_{3}$ we have 
\begin{equation}
\vartheta :=\left\vert p_{n}\right\vert +L_{\gamma
^{+}}L_{d}\sum\limits_{j=n}^{\infty }\left( \left\vert \frac{1}{r_{j}}%
\right\vert \sum\limits_{i=j}^{\infty }\left\vert a_{i}\right\vert \right) ^{%
\frac{1}{\gamma ^{+}}}+L_{\gamma ^{+}}L_{\alpha }\sum\limits_{j=n}^{\infty
}\left( \left\vert \frac{1}{r_{j}}\right\vert \sum\limits_{i=j}^{\infty
}\left\vert q_{i}\right\vert \right) ^{\frac{1}{\gamma ^{+}}}<1.
\label{sssss}
\end{equation}%
Let us denote 
\begin{equation*}
\limsup_{n\rightarrow \infty }\left\vert x_{n}-y_{n}\right\vert =l,
\end{equation*}%
and observe that 
\begin{equation*}
\limsup_{n\rightarrow \infty }\left\vert x_{n}-y_{n}\right\vert
=\limsup_{n\rightarrow \infty }\left\vert x_{n-k}-y_{n-k}\right\vert
=\limsup_{n\rightarrow \infty }\left\vert x_{n+1}-y_{n+1}\right\vert .
\end{equation*}%
Thus from the above we have 
\begin{equation*}
l\leq \vartheta \cdot l.
\end{equation*}%
This means that $\limsup\limits_{n\rightarrow \infty }\left\vert
x_{n}-y_{n}\right\vert =0$. This completes the proof since for $\varepsilon
>0$ there exists $n_{4}\in N_{0}$ such that for every $n\geq n_{4}\geq n_{3}$
the following inequality holds 
\begin{equation*}
\left\vert x_{n}-y_{n}\right\vert \leq \varepsilon .
\end{equation*}
\end{proof}

\begin{example}
Let us consider the difference equation 
\begin{equation*}
\Delta \left\{ 3^{-\frac{1}{2n+1}}(-1)^{n}\left[ \Delta (x_{n}+\frac{1}{2}%
x_{n-2})\right] ^{\frac{1}{2n+1}}\right\} +\frac{1}{2^{n}}\left(
x_{n}\right) ^{5}+\frac{1}{2^{n}}(x_{n+1})^{3}=0.
\end{equation*}%
Here $\alpha =5,$ $p_{n}=\frac{1}{2}$, $k=2$, $r_{n}=3^{-\frac{1}{2n+1}%
}(-1)^{n}$, $\gamma _{n}=\frac{1}{2n+1}$, $a_{n}=\frac{1}{2^{n}}$, $q_{n}=%
\frac{1}{2^{n}}$, and $f(x)=x^{3}$. All assumptions of Theorem \ref{T1} are
satisfied with $\gamma ^{+}=1$ and $\lim\limits_{n\rightarrow \infty }p_{n}=%
\frac{1}{2}$. Hence, there exists a bounded solution of the above equation.
In fact, sequence $x_{n}=(-1)^{n}$ is one of such solutions.
\end{example}

In the following example we present the equation which coefficients fulfil
the assumptions of Theorem \ref{T1} and for which one unbounded solution is
known. With our method we get the second bounded solution.

\begin{example}
Let us consider the difference equation 
\begin{equation*}
\Delta \left[ (n-1)\Delta \left( x_{n}+\frac{1}{(n-1)^{2}(n-2)}%
x_{n-1}\right) \right] -\frac{4}{n^{2}(n-1)^{3}}(x_{n})^{3}
\end{equation*}%
\begin{equation*}
-\frac{1}{n(n+1)\sin \left[ n(n+1)\right] }\sin (x_{n+1})=0,n\geq 3.
\end{equation*}%
Here $p_{n}=\frac{1}{(n-1)^{2}(n-2)}$, $k=1$, $r_{n}=n-1$, $\alpha =3$, $%
\gamma _{n}=1$, $a_{n}=-\frac{1}{n(n+1)\sin \left[ n(n+1)\right] }$, $q_{n}=-%
\frac{4}{n^{2}(n-1)^{3}}$, and $f(x)=\sin x$. All assumptions of Theorem \ref%
{T1} are satisfied with $\lim\limits_{n\rightarrow \infty }p_{n}=0$.
Sequence $x_{n}=n(n-1)$ is a solution of the above equation. But this
solution is unbounded. On the virtue of Theorem \ref{T1} we know that there
exists another solution of the above equation which is bounded.
\end{example}

\section{Uniqueness}

In this section we undertake the question of the uniqueness of solutions to %
\eqref{e0}. It appears that assumptions leading to the type of a stability
which we considered also provide that any solution to \eqref{e0} is
eventually unique, i.e. there exists $n_{4}\geq n_{3}$ such that for any $%
n\geq n_{4}$ all solutions coincide.

\begin{theorem}
\label{Theo_uniqueness}With assumptions of Theorem \ref{T1} equation %
\eqref{e0} has a solution which is eventually unique, i.e. there exists $%
n_{4}\geq n_{3}$ such that for any $n\geq n_{4}$ and any bounded solution $y$
to \eqref{e0} it follows that $x_{n}=y_{n}.$
\end{theorem}

\begin{proof}
We proceed as in the proof of Theorem \ref{T1} with same notation.

We have that $0<\vartheta <1$ for $n\geq n_{4}$, where $\vartheta $ is defined by \eqref{sssss}. Let us denote 
\begin{equation*}
\sup\limits_{n\in {\mathbb{N}}_{n_4}}\left\vert x_{n}-y_{n}\right\vert =l^*.
\end{equation*}%
Hence, from \eqref{arg}, we get 
\begin{equation*}
l^* \leq \vartheta \cdot l^*.
\end{equation*}%
This means that $\sup\limits_{n\geq n_{4}}\left\vert x_{n}-y_{n}\right\vert
=0$ and $x_{n}=y_{n}$ for $n\geq n_{4}$.
\end{proof}

\section{Continuous dependence on parameters}

In this section we consider the case when the nonlinear term in the given
problem depends on a convergent sequence on parameters and we investigate
what happens when we approach a limit with this sequence. We assume that $%
g:R\rightarrow R$ is a continuous function and we consider a family of
problems for $m=1,2,...$ 
\begin{equation}
\Delta \left( r_{n}\left( \Delta \left( x_{n}+p_{n}x_{n-k}\right) \right)
^{\gamma _{n}}\right) +q_{n}x_{n}^{\alpha }+a_{n}f(x_{n+1})g\left(
u^{m}\right) =0.  \label{family}
\end{equation}%
Here a sequence $\left( u^{m}\right) \subset l^{\infty }$ is convergent to
some $u^{0}\in l^{\infty }$ and $u^{m}$ denotes the $m-th$ member of $\left(
u^{m}\right) $. Our result is as follows:

\begin{theorem}
\label{T1_dep}Let the sequence $(\gamma _{n})$ of the ratio of odd positive
integers be such that \eqref{gamma} holds. Let an integer $k$ be fixed and
let $\alpha =1$. Let $\left( u^{m}\right) \subset l^{\infty }$ be a
convergent sequence of parameters; let $u^{0}\in l^{\infty }$ be its limit.
Assume further for any $d_{1}>0$ that there exists a positive constant $%
D_{1} $ such that 
\begin{equation}
\left\vert f\left( x\right) g(s)-f\left( y\right) g(t)\right\vert \leq
D_{1}\left\vert s-t\right\vert  \label{growth_g}
\end{equation}%
for any $s,t\in R$ and $x,y\in \left[ -d_{1},d_{1}\right] $. Assume next
that for any $d_{2}>0$ that there exists a positive constant $D_{2}$ such
that 
\begin{equation*}
\left\vert f\left( x\right) g(s)-f\left( y\right) g(t)\right\vert \leq
D_{2}\left\vert x-y\right\vert
\end{equation*}%
for any $s,t\in \left[ -d_{2},d_{2}\right] $, $x,y\in R.$ Moreover assume
that conditions \eqref{z2}--\eqref{add_series} hold. Then for each $%
m=1,2,... $ equation~\eqref{family} has an one asymptotically stable
solution $x^{m}:N_{k}\rightarrow R$ corresponding to $u=u^{m}$. Also
equation~\eqref{family} has an asymptotically stable solution $%
x^{0}:N_{k}\rightarrow R$ corresponding to $u=u^{0}$. Moreover, sequence $%
\left( x^{m}\right) \subset l^{\infty }$ is convergent to $x^{0}$.
\end{theorem}

\begin{proof}
Let $\left\Vert \cdot \right\Vert $ denotes the classical norm in $l^{\infty
}$. Existence of a sequence $\left( x^{m}\right) $ follows by Theorem \ref%
{L2}. Let us take any $\varepsilon >0$ and chose $s\geq t$, $s,t\in N$ so
that $\left\Vert u^{s}-u^{t}\right\Vert <\varepsilon $ which is possible
since $\left( u^{m}\right) $ is convergent and so it satisfies the Cauchy
condition. As for the convergence of the sequence $\left( x^{m}\right) $ we
see from \eqref{arg}, the definition of $B$ and \eqref{growth_g} that for
the chosen above $s\geq t$ and a suitable large $n\geq n_{3}$ 
\begin{equation*}
\begin{array}{l}
\left\vert x_{n}^{s}-x_{n}^{t}\right\vert =\left\vert
(Tx^{s})_{n}-(Tx^{t})_{n}\right\vert \,\bigskip  \\ 
\leq \left\vert p_{n}\right\vert \left\vert x_{n}^{s}-x_{n}^{t}\right\vert
+L_{\gamma ^{+}}D_{1}\sum\limits_{j=n}^{\infty }\left\vert \frac{1}{r_{j}}%
\right\vert ^{\frac{1}{\gamma ^{+}}}\sum\limits_{i=j}^{\infty }\left\vert
a_{i}\right\vert \left\vert u^{s}-u^{t}\right\vert \,\bigskip  \\ 
+L_{\gamma ^{+}}\sum\limits_{j=n}^{\infty }\left\vert \frac{1}{r_{j}}%
\right\vert ^{\frac{1}{\gamma ^{+}}}\sum\limits_{i=j}^{\infty
}q_{i}\left\vert x_{i}^{s}-y_{i}^{t}\right\vert .%
\end{array}%
\end{equation*}%
Note that for $n$ large enough, say $n\geq n_{4}\geq n_{3}$ we have 
\begin{equation*}
0<\vartheta _{1}:=\sup\limits_{n\geq n_{4}}\left( \left\vert
p_{n}\right\vert +L_{\gamma ^{+}}\sum\limits_{j=n}^{\infty }\left\vert \frac{%
1}{r_{j}}\right\vert ^{\frac{1}{\gamma ^{+}}}\sum\limits_{i=j}^{\infty
}q_{i}\right) <1
\end{equation*}%
and 
\begin{equation*}
0<\vartheta _{2}:=\sup\limits_{n\geq n_{4}}\left( L_{\gamma
^{+}}D_{1}\sum\limits_{j=n}^{\infty }\left\vert \frac{1}{r_{j}}\right\vert ^{%
\frac{1}{\gamma ^{+}}}\sum\limits_{i=j}^{\infty }\left\vert a_{i}\right\vert
\right) <1.
\end{equation*}%
Denote $\left\Vert x^{s}-x^{t}\right\Vert =\sup\limits_{n\geq
n_{4}}\left\vert x_{n}^{s}-x_{n}^{t}\right\vert $. This means that 
\begin{equation*}
\left\Vert x^{s}-x^{t}\right\Vert \leq \vartheta _{1}\left\Vert
x^{s}-x^{t}\right\Vert +\vartheta _{2}\left\Vert u^{s}-u^{t}\right\Vert .
\end{equation*}%
It follows that $\left( x^{m}\right) $ is a Cauchy sequence. We denote its
limit by $x^{0}$. Let $x^{1}$ be solution corresponding to $u^{0}$. By
Theorem \ref{Theo_uniqueness} it follows that eventually $x^{0}=x^{1}$.
\end{proof}

\section{Final comments}

As mentioned in the Introduction our results can be applied for
Sturm--Liouville difference equations. We investigate the following equation 
\begin{equation}
\Delta \left( r_{n}\left( \Delta x_{n}\right) ^{\gamma _{n}}\right)
+a_{n}x_{n+1}=0.  \label{S-L}
\end{equation}%
Applying same ideas as in the proof of Theorem \ref{L2} we get the following

\begin{corollary}
Let the sequence $(\gamma _{n})$ of the ratio of odd positive integers be
such that \eqref{gamma} holds. Assume that the sequences $r:N_{0}\rightarrow
R\setminus \{0\}$, $a:N_{0}\rightarrow R$, $\sum\limits_{i=0}^{\infty
}\left\vert a_{i}\right\vert <+\infty $, and conditions \eqref{z2} is
satisfied. Then, there exists a bounded solution $x:N_{0}\rightarrow R$ of
equation~\eqref{S-L}.
\end{corollary}

\begin{example}
Concerning the nonlinear terms for which our results are applicable we see
that in Theorem \ref{L2} any $C^{1}$ function $f$ satisfies the requirement
for being locally Lipschitz. As far and Theorems \ref{T1} and \ref%
{Theo_uniqueness} linear function $f$ could be used, namely: $f\left(
x\right) =\alpha x.$ In Theorem \ref{T1_dep} once can consider $f\left(
x\right) =\alpha x$ and $g\left( u\right) =u.$
\end{example}

\section*{Acknowledgement}

The Authors would like to thank anonymous Referees for suggestions which
helped to improve both the presentation and the quality of results contained
in this paper.

\end{document}